\documentclass[11pt,reqno]{amsart}

\usepackage[T1]{fontenc}
\usepackage[utf8]{inputenc}
\usepackage[english]{babel}
\usepackage{lmodern}
\usepackage{microtype}
\usepackage{amsmath,amssymb,amsthm,mathtools}
\usepackage{xcolor}
\usepackage{tikz}
\usepackage{hyperref}

\newtheorem{theorem}{Theorem}[section]
\newtheorem{proposition}[theorem]{Proposition}
\newtheorem{corollary}[theorem]{Corollary}
\newtheorem{lemma}[theorem]{Lemma}
\newtheorem{conjecture}[theorem]{Conjecture}
\newtheorem*{introtheorem}{Theorem}
\theoremstyle{definition}
\newtheorem{example}[theorem]{Example}
\newtheorem{question}[theorem]{Question}
\newtheorem{remark}[theorem]{Remark}

\DeclareMathOperator{\code}{code}
\newcommand{\setN}{\mathbb{N}}
\newcommand{\setZ}{\mathbb{Z}}
\newcommand{\xvec}{\mathbf{x}}
\newcommand{\symS}{S}
\newcommand{\schubertS}{\mathfrak{S}}
\newcommand{\keypoly}{\kappa}
\newcommand{\stretchop}{\mathbin{\ast}}
\newcommand{\stretchperm}[2]{#1\stretchop #2}
\newcommand{\defin}[1]{%
  \relax\ifmmode%
  #1%
  \else\emph{#1}%
  \fi%
}
\newcommand{\rotheDiagramWithBlocks}[2][]{%
\begin{tikzpicture}[y=-1cm,baseline=(current bounding box.south),#1]
  \def\perm{#2}
  \def\rotheN{0}
  \foreach \val [count=\i] in \perm {\xdef\rotheN{\i}}
  \draw[thin,gray!30] (0,0) grid (\rotheN,\rotheN);
  \foreach \wi [count=\i] in \perm {
    \foreach \wj [count=\j] in \perm {
      \ifnum\i<\j
        \ifnum\wi>\wj
          \def\cellfill{gray!25}
          \ifnum\wj>1
            \ifnum\wj<5
              \def\cellfill{blue!18}
            \fi
          \fi
          \ifnum\wj>6
            \ifnum\wj<9
              \def\cellfill{green!18}
            \fi
          \fi
          \draw[thick,fill=\cellfill,fill opacity=.82]
            (\wj-1,\i-1) rectangle (\wj,\i);
        \fi
      \fi
    }
  }
  \foreach \wi [count=\i] in \perm {
    \draw[line width=1.2pt]
      (\wi-.5,\i-.5) -- (\rotheN,\i-.5);
    \draw[line width=1.2pt]
      (\wi-.5,\i-.5) -- (\wi-.5,\rotheN);
    \fill (\wi-.5,\i-.5) circle (.15);
  }
  \draw[very thick] (0,0) rectangle (\rotheN,\rotheN);
  \foreach \i in {1,...,\rotheN} {
    \node[left,font=\tiny] at (0,\i-.5) {$\i$};
    \node[above,font=\tiny] at (\i-.5,0) {$\i$};
  }
\end{tikzpicture}%
}

\title[Stretched Schubert Polynomiality and Stability]
{Polynomiality of Stretched Schubert Structure Constants and Stable Key
Expansions}
\author{Per Alexandersson}
\address{Department of Mathematics, Stockholm University,
  SE-106 91 Stockholm, Sweden}
\email{per.w.alexandersson@gmail.com}
\date{}

\begin{document}

\begin{abstract}
We prove that monomial coefficients in affine families of key and Schubert
polynomials are eventually polynomial.  These results extend to finite
products.  Using M.~Watanabe's Schubert duality, we deduce that stretched
Schubert structure constants are eventually polynomial.  This proves a
conjecture stated by I.~Pak and Z.~Slonim.  We also prove that the
positive key expansion of every affine-code family of Schubert polynomials
eventually has fixed coefficients and affine indices.
\end{abstract}

\maketitle

\section{Introduction}
\label{sec:introduction}

Schubert polynomials form a distinguished basis of the polynomial ring, and
their multiplication is governed by the \defin{Schubert structure constants}.
For permutations \(u,v\), we define
\(\defin{c_{u,v}^w}\) by
\[
\schubertS_u\schubertS_v
=\sum_w c_{u,v}^w\schubertS_w.
\]
Let \(\defin{\code(u)}\) denote the Lehmer code of \(u\), and let
\(\defin{\ell(u)}\) denote its inversion number.  For a positive integer
\(N\), we define
\[
\defin{\stretchperm{N}{u}}\coloneqq\code^{-1}\bigl(N\code(u)\bigr).
\]
If \(N\code(u)=(c_1,\dotsc,c_h)\), we append \(t\) trailing zeroes with
\(c_i\leq h+t-i\) for all \(i\); the resulting finite sequence is then a
Lehmer code.  Increasing \(t\) only appends fixed points to the corresponding
permutation.
I.~Pak and Z.~Slonim proved that the resulting stretched Schubert structure
constants are eventually quasipolynomial~\cite{PakSlonim2026x}, and
restated A.~N.~Kirillov's earlier question as the conjecture that the period
is always one~\cite[Sec.~7.1]{PakSlonim2026x}.  We prove this conjecture by first
studying affine monomial coefficients of stretched Schubert polynomials.  We
also prove that their positive key expansions eventually have fixed
coefficients and affine indices.

For
\[
\defin{\xvec}\coloneqq(x_1,\dotsc,x_h),
\qquad
\defin{\xvec^a}\coloneqq x_1^{a_1}\dotsm x_h^{a_h}
\quad\text{for }a\in\setZ^h,
\]
we write \(\defin{[\xvec^a]f}\) for the coefficient of the monomial
\(\xvec^a\) in \(f\).
Classical Kostka numbers are Schur monomial coefficients.  Following the
Schubert--Kostka matrix terminology of A.~Postnikov and R.~Stanley
\cite[Sec.~17]{PostnikovStanley2008}, we call the analogous monomial
coefficients of Schubert polynomials Schubert--Kostka coefficients; similarly,
we call the monomial coefficients of key polynomials key--Kostka coefficients.
We write
\(\defin{\setN}\coloneqq\{0,1,2,\dotsc\}\).  A function of \(N\) is
\defin{eventually polynomial} if there is a polynomial \(P(N)\) such that the
function agrees with \(P(N)\) for all \(N\gg0\).

The remaining definitions are recalled below.  We begin with the analogous
monomial-coefficient question for key polynomials, or type~\(A\) Demazure
characters.

\begin{introtheorem}[{\hyperref[cor:pureKeyKostka]
  {see Corollary~\ref*{cor:pureKeyKostka}}}]
  Let \(\alpha,\nu\) be weak compositions of the same total degree.
  Then the map
\[
N\longmapsto
[\xvec^{N\nu}]\keypoly_{N\alpha}(\xvec)
\]
is a polynomial in \(N\) for positive integers \(N\).
\end{introtheorem}

Questions of Ehrhart positivity for Demazure characters were considered in
earlier work with E.~Alhajjar~\cite{AlexanderssonAlhajjar2018}.
Nonnegativity of the coefficients remains open even for classical stretched
Littlewood--Richardson and Kostka polynomials~\cite{KingTolluToumazet2004}.
We propose a column-monotonicity strengthening of the Kostka case in
Conjecture~\ref{conj:kostkaColumnMonotonicity}.
Another recent approach to stretched Kostka
coefficients, using partition-division maps on symmetric functions, is given
in~\cite{AlexanderssonDai2026x}.

The Schubert analogue is eventually polynomial along every affine weight ray.

\begin{introtheorem}[{\hyperref[thm:affineSchubertKostka]
  {see Corollary~\ref*{thm:affineSchubertKostka}}}]
Fix a permutation \(u\) and an integer \(h\) such that
\(\code(u)\in\setN^h\).  Let \(\nu,\eta\in\setZ^h\), and suppose that
\(N\nu+\eta\) is a weak composition and has total degree \(N\ell(u)\) for all
sufficiently large \(N\).  Then
\[
N\longmapsto
[\xvec^{N\nu+\eta}]\schubertS_{\stretchperm{N}{u}}(\xvec)
\]
is eventually polynomial.
\end{introtheorem}

Coefficientwise polynomiality does not control the number of keys in the
expansion of a Schubert polynomial.  The following theorem gives a stronger
stability under stretching.

\begin{introtheorem}[{\hyperref[thm:stableAffineKeyExpansion]
  {see Theorem~\ref*{thm:stableAffineKeyExpansion}}}]
For every permutation \(u\), there are positive integers
\(c_1,\dotsc,c_t\), weak compositions
\(\gamma_1,\dotsc,\gamma_t\), and integer vectors
\(\delta_1,\dotsc,\delta_t\), all independent of \(N\), such that
\[
\schubertS_{\stretchperm{N}{u}}
=
\sum_{j=1}^t c_j\keypoly_{N\gamma_j+\delta_j}
\qquad(N\gg0).
\]
\end{introtheorem}

Thus the number of terms and their positive coefficients eventually
stabilize, while every key index varies affinely.  The more general
Theorem~\ref{thm:stableAffineKeyExpansion} allows the Lehmer code to be any
affine function of \(N\).

Schubert duality gives the polynomiality result for structure constants.

\begin{introtheorem}[{\hyperref[cor:stretchedSchubertStructureConstants]
  {see Corollary~\ref*{cor:stretchedSchubertStructureConstants}}}]
For fixed permutations \(u,v,w\), the stretched Schubert structure constant
\[
N\longmapsto
c_{\stretchperm{N}{u},\stretchperm{N}{v}}^{\stretchperm{N}{w}}
\]
is eventually polynomial.
\end{introtheorem}

For key polynomials, Demazure operators reduce coefficient extraction to
type~\(A\) vector partition functions.  For Schubert polynomials, sorting the
first \(h\) values of an affine-code permutation gives a dominant permutation,
so the Schubert polynomial is a fixed product of divided differences applied
to a monomial with affine exponent.  This gives Schubert--Kostka polynomiality
directly.  Moving polynomial multipliers through the Demazure operators and
using key bases on blocks of equal slopes gives a finite signed affine key
expansion.  The known Schubert-to-key positivity makes its coefficients
nonnegative.  M.~Watanabe's duality formula gives the result for structure
constants~\cite[Lem.~7.4]{Watanabe2016}.

\looseness=1
The paper is organized as follows.  Section~\ref{sec:dividedDifferenceLemma}
gives the vector-partition-function lemma, which we apply to key polynomials
in Section~\ref{sec:keyKostka}.  Section~\ref{sec:stretchingRotheDiagrams}
describes affine families of Lehmer codes, and Section~\ref{sec:dominantSorting}
writes their Schubert polynomials as fixed divided-difference expressions.
Section~\ref{sec:stableKeyExpansions} proves Schubert--Kostka polynomiality and
stable key expansions.  Finally, Section~\ref{sec:products}
treats products and Schubert structure constants.

\section{A divided-difference lemma}
\label{sec:dividedDifferenceLemma}

We write
\(\defin{s_i}\) for the transposition of \(x_i\) and \(x_{i+1}\).  Let
\[
\defin{\pi_i(f)}
\coloneqq
\frac{x_i f-x_{i+1}s_i(f)}{x_i-x_{i+1}}.
\]
The operator \(\pi_i\) is the \defin{isobaric divided difference}, or
\defin{Demazure operator}.

\begin{lemma}[Demazure expressions]
\label{lem:fixedDemazureExpressions}
Let \(h\in\setN\).  Suppose that, for all sufficiently large \(N\), the Laurent
polynomial \(F_N\) is obtained from \(1\) by a fixed finite sequence of the
following operations:
\begin{enumerate}
\item multiply by a Laurent monomial
  \(\xvec^{Na+b}\), where \(a,b\in\setZ^h\);
\item apply a Demazure operator \(\pi_i\).
\end{enumerate}
Then, for any \(\nu,\eta\in\setZ^h\), the map
\[
N\longmapsto
[\xvec^{N\nu+\eta}]F_N
\]
is eventually polynomial in \(N\).  The same conclusion holds for a product
of finitely many such expressions.
\end{lemma}

\begin{proof}
Let \(\defin{\alpha_i}\) be the \(i\)-th simple root of type~\(A\).  Thus
\(\alpha_i\) has a \(1\) in position \(i\), a \(-1\) in position \(i+1\),
and zeroes elsewhere; it is the difference of two consecutive unit vectors.
As an identity of rational functions, we have
\[
\pi_i(f)
=
\frac{f}{1-\xvec^{-\alpha_i}}
-
\frac{\xvec^{-\alpha_i}s_i(f)}{1-\xvec^{-\alpha_i}}.
\]
Starting with a Laurent monomial and iterating this identity gives a finite
sum of terms
\begin{equation}
\label{eq:rootDenominatorExpansion}
C\xvec^{N\beta+\gamma}
\prod_{\rho\in\Phi}
\frac{1}{1-\xvec^{-\rho}},
\end{equation}
where \(C\in\setZ\), the vectors \(\beta,\gamma\in\setZ^h\) are fixed, and
\(\Phi\) is a finite \emph{multiset} of type~\(A\) roots.  If a denominator
contains a negative root, we use
\[
\frac{1}{1-\xvec^\rho}
=-
\frac{\xvec^{-\rho}}{1-\xvec^{-\rho}}
\]
to make it positive.  Thus all denominator roots in
\eqref{eq:rootDenominatorExpansion} may be taken positive.
We now expand every denominator as a geometric series in
\(\xvec^{-\rho}\).  These expansions use the pointed cone generated
by \(-\alpha_1,\dotsc,-\alpha_{h-1}\), so each monomial receives only finitely many
contributions.  Coefficient extraction is therefore well-defined and agrees
with that of the Laurent polynomial \(F_N\).

\begingroup
For a finite multiset \(\Phi\) of positive type~\(A\) roots and \(b\in\setZ^h\), let
\[
\defin{K_\Phi(b)}
\coloneqq
\#\Bigl\{(m_\rho)_{\rho\in\Phi}\in\setN^\Phi:\sum_{\rho\in\Phi}m_\rho\rho=b\Bigr\}
\]
be the Kostant partition function of \(\Phi\).  The coefficient of
\(\xvec^{N\nu+\eta}\) in one such term is
\(C\,K_\Phi\bigl(N(\beta-\nu)+(\gamma-\eta)\bigr)\).
By Lemma~\ref{lem:kostantRays} below, each term, and therefore their finite sum, is
eventually polynomial.

Multiplying finitely many expressions merely concatenates the multisets of
roots in the denominators, so the same argument proves the product statement.
\endgroup
\end{proof}

\begingroup
It remains to prove that Kostant partition functions are eventually polynomial
along affine rays.  We regard \(\Phi\) as a directed multigraph on
\(\{1,\dotsc,h\}\), with one edge \(i\to j\) for each copy of the root
\(e_i-e_j\) in \(\Phi\).  A set \(S\subseteq\{1,\dotsc,h\}\) is
\defin{closed} if \(i\in S\) implies \(j\in S\) for every edge \(i\to j\).
Let \(\defin{\setN\Phi}\) be the additive monoid generated by \(\Phi\).  The
following feasibility criterion is the uncapacitated case of D.~Gale's theorem
on flows in networks~\cite{Gale1957}; we include the short proof.

\begin{lemma}[Feasible flows]
\label{lem:gale}
An integer vector \(b\in\setZ^h\) lies in \(\setN\Phi\) if and only if
\(\sum_i b_i=0\) and \(\sum_{i\in S}b_i\leq0\) for every closed set \(S\).
\end{lemma}

\begin{proof}
For a closed set \(S\), each root of \(\Phi\) contributes \(0\) or \(-1\) to
\(\sum_{i\in S}b_i\), which proves necessity.  For sufficiency we use induction
on \(\sum_i\max(b_i,0)\).  If this sum is zero, then \(b=0\).  Otherwise choose
\(p\) with \(b_p>0\).  Call a closed set \(S\) \defin{tight} if
\(\sum_{i\in S}b_i=0\).  If \(A\) and \(B\) are tight, then \(A\cup B\) and
\(A\cap B\) are closed, so the sums over them are at most \(0\); they add up to
\(0\), so both vanish and \(A\cup B\) is tight.  Let \(U\) be the union of all
tight sets not containing \(p\); it is tight.  Let \(R\) be the set of vertices
reachable from \(p\).  Then \(R\cup U\) is closed, so \(\sum_{i\in R\setminus U}b_i\leq0\), and since \(b_p>0\) there is
\(k\in R\setminus U\) with \(b_k<0\).  A path from \(p\) to \(k\) shows that
\(e_p-e_k\in\setN\Phi\).  The vector \(b'=b-e_p+e_k\) again satisfies the
conditions: a closed set containing \(p\) contains \(k\), and a closed set \(S\)
with \(p\notin S\) and \(\sum_{i\in S}b_i=0\) is contained in \(U\), so
\(k\notin S\).  By induction \(b'\in\setN\Phi\), and hence \(b\in\setN\Phi\).
\end{proof}

\begin{lemma}[Partition functions along rays]
\label{lem:kostantRays}
Let \(\Phi\) be a finite multiset of positive type~\(A\) roots, and let
\(v,w\in\setZ^h\).  Then \(N\mapsto K_\Phi(Nv+w)\) is eventually polynomial.
\end{lemma}

\begin{proof}
We use induction on \(|\Phi|\).  For \(\Phi=\emptyset\), \(K_\emptyset(Nv+w)\)
is \(1\) if \(Nv+w=0\) and \(0\) otherwise, which is eventually constant.

Suppose that \(v\notin\setN\Phi\).  By Lemma~\ref{lem:gale}, either
\(\sum_iv_i\neq0\) or \(\sum_{i\in S}v_i>0\) for some closed set \(S\).  Every
\(b\in\setN\Phi\) satisfies \(\sum_ib_i=0\) and \(\sum_{i\in S}b_i\leq0\), and
on \(b=Nv+w\) these are affine functions of \(N\) that eventually have the wrong
values.  Hence \(K_\Phi(Nv+w)=0\) for \(N\gg0\).

Otherwise \(v=\rho_1+\dotsb+\rho_L\) with \(\rho_t\in\Phi\).  Splitting according
to whether a fixed copy of \(\rho\in\Phi\) is used gives
\(K_\Phi(b)=K_\Phi(b-\rho)+K_{\Phi\setminus\rho}(b)\).  Telescoping along
\(\rho_1,\dotsc,\rho_L\),
\[
K_\Phi(b+v)-K_\Phi(b)
=\sum_{t=1}^{L}K_{\Phi\setminus\rho_t}(b+\rho_1+\dotsb+\rho_t).
\]
For \(b=Nv+w\), each summand is eventually polynomial in \(N\) by induction.
Hence the forward difference of \(N\mapsto K_\Phi(Nv+w)\) is eventually
polynomial, and so is the function itself.
\end{proof}

\begin{remark}
Lemma~\ref{lem:kostantRays} also follows from B.~Sturmfels' chamber
theorem~\cite{Sturmfels1995}.  The matrix with columns \(\Phi\) is totally
unimodular, so its vector partition function agrees with a polynomial on the
closure of each chamber; the polynomials of adjacent chambers agree on the
intersection of their closures~\cite[Sec.~4]{DeLoeraSturmfels2003}.  The chamber
complex is a finite fan, so for \(N\gg0\) the points \(Nv+w\) lie in one fixed
cell of it, and hence in the closure of one chamber, or they lie outside the
cone generated by \(\Phi\), where \(K_\Phi\) vanishes.  In the argument above,
the integrality behind polynomiality enters only through Lemma~\ref{lem:gale}.
\end{remark}
\endgroup

\section{Key--Kostka coefficients}
\label{sec:keyKostka}

Key polynomials, introduced by A.~Lascoux and M.-P.~Sch\"utzenberger
\cite{Lascoux1990Keys}, are the type~\(A\) Demazure characters.  Combinatorial
models include S.~K.~Mason's semi-skyline augmented fillings, which connect
nonsymmetric Macdonald polynomials with Demazure atoms and operators
\cite{Mason2009}, and the Kohnert tableaux of S.~H.~Assaf and D.~Searles
\cite{AssafSearles2018,AssafSearles2019}.

A \defin{weak composition} is a finite sequence of nonnegative integers; we
append trailing zeroes whenever necessary.  We use the divided-difference
definition of the key polynomial \(\keypoly_\alpha\).  If \(\alpha\) is weakly
decreasing, then
\[
\defin{\keypoly_\alpha}\coloneqq\xvec^\alpha.
\]
If \(\alpha_i<\alpha_{i+1}\), then
\[
\defin{\keypoly_\alpha}\coloneqq\pi_i\keypoly_{s_i\alpha}.
\]
This recursion is well-founded: each such step decreases by one the number of
pairs \(j<k\) with \(\alpha_j<\alpha_k\).
For fixed \(h\), the family
\[
\{\keypoly_\alpha:\alpha\in\setN^h\}
\]
is a basis of \(\mathbb Q[x_1,\dotsc,x_h]\).
For a permutation \(w\) with a reduced decomposition
\(w=s_{i_1}\dotsm s_{i_\ell}\), define
\[
\defin{\pi_w}
\coloneqq
\pi_{i_1}\dotsm\pi_{i_\ell}.
\]
The Demazure relations imply that \(\pi_w\) is independent of the chosen
reduced decomposition.
Equivalently, if \(\lambda\) is the decreasing rearrangement of \(\alpha\)
and \(w\) is the shortest permutation satisfying \(w\lambda=\alpha\), then
\(\keypoly_\alpha=\pi_w\xvec^\lambda\).

For an affine family of indices, the relative order of the parts eventually
stops changing.  We can then use one Demazure word for the entire family.

\begin{theorem}[Affine key--Kostka polynomiality]
\label{thm:affineKeyKostka}
Fix \(\alpha,\nu\in\setN^h\) and \(\beta,\eta\in\setZ^h\).  Suppose that
\(N\alpha+\beta\) and
\(N\nu+\eta\) are weak compositions of the same total degree for all
sufficiently large \(N\).  Then
\[
N\longmapsto
[\xvec^{N\nu+\eta}]\keypoly_{N\alpha+\beta}(\xvec)
\]
is eventually polynomial.
\end{theorem}

\begin{proof}
For \(N\gg0\), the relative order of the parts of \(N\alpha+\beta\) is fixed.
Its decreasing rearrangement therefore has the form \(N\lambda+\lambda_0\),
where \(\lambda,\lambda_0\) are fixed, and the shortest permutation \(w\) satisfying
\[
w(N\lambda+\lambda_0)=N\alpha+\beta
\]
is fixed.  Choose a reduced word for \(w\).  The expression
\[
\keypoly_{N\alpha+\beta}
=\pi_w\xvec^{N\lambda+\lambda_0}
\]
now satisfies Lemma~\ref{lem:fixedDemazureExpressions}, which proves the
claim.
\end{proof}

\begin{corollary}
\label{cor:pureKeyKostka}
For weak compositions \(\alpha,\nu\) of the same total degree,
\[
N\longmapsto
[\xvec^{N\nu}]\keypoly_{N\alpha}(\xvec),
\qquad N\geq1,
\]
is a polynomial in \(N\).
\end{corollary}

\begin{proof}
Theorem~\ref{thm:affineKeyKostka}, with \(\beta=\eta=0\), gives a polynomial
for \(N\gg0\).  The Kogan-face model for Demazure characters expresses the
key as a weight generating function of a union of faces of a
Gelfand--Tsetlin polytope
\cite{Kogan2005,KiritchenkoSmirnovTimorin2010}; see also
\cite{AlexanderssonAlhajjar2018}.  Intersecting these faces with the equations
that prescribe the weight \(\nu\) gives a finite union of rational
polytopes.  For \(N\geq1\), scaling the key index and weight by \(N\)
dilates this union by \(N\).  Inclusion--exclusion and Ehrhart's theorem
therefore make the coefficient a quasipolynomial on the positive integers.
A quasipolynomial that agrees with one polynomial for all sufficiently large
\(N\) agrees with that polynomial for every positive \(N\).  The claim
follows.
\end{proof}

At \(N=0\), we use the polynomial continuation from positive stretching
factors, as in Ehrhart theory.  This need not agree with substituting zero
directly into both indices: an empty weight slice has zero Ehrhart function,
whereas \([\xvec^0]\keypoly_0=1\).

\begin{proposition}[Products of affine keys]
\label{prop:productsAffineKeys}
Fix \(h\), and let \(\alpha_j\in\setN^h\) and
\(\beta_j\in\setZ^h\), for \(1\leq j\leq s\), such that
\(N\alpha_j+\beta_j\) is a weak composition for \(N\gg0\).  If
\(\nu\in\setN^h\), \(\eta\in\setZ^h\), and \(N\nu+\eta\) is eventually
nonnegative, then
\[
[\xvec^{N\nu+\eta}]
\prod_{j=1}^s\keypoly_{N\alpha_j+\beta_j}(\xvec)
\]
is eventually polynomial.
\end{proposition}

\begin{proof}
Use the fixed reduced word from the proof of
Theorem~\ref{thm:affineKeyKostka} for each factor, and apply the product part
of Lemma~\ref{lem:fixedDemazureExpressions}.  The claim follows.
\end{proof}

Theorem~\ref{thm:affineKeyKostka} extends the polynomiality result from joint
work with E.~Kantarc{\i} O{\u{g}}uz from stretched row-flagged skew Schur
polynomials and cylindric Schur functions to affine families of key
polynomials~\cite{AlexanderssonOguz2023x}.

\begingroup
Conversely, Lemma~\ref{lem:fixedDemazureExpressions} recovers that result
directly.  Let \(\lambda/\mu\) be a skew shape with \(m\) rows, and let
\(\mathbf{a},\mathbf{b}\) be weakly increasing flags with entries in
\(\{1,\dotsc,h\}\).  The \defin{row-flagged skew Schur polynomial}
\(s_{\lambda/\mu}(\mathbf{a},\mathbf{b})\) is the sum of \(\xvec^T\) over the
semistandard Young tableaux \(T\) of shape \(\lambda/\mu\) whose entries in
row \(i\) lie in \(\{a_i,\dotsc,b_i\}\).  For \(r\in\setZ\), let
\(\mathrm{h}_r(a,b)\) be the complete homogeneous symmetric polynomial of
degree \(r\) in \(x_a,\dotsc,x_b\), with \(\mathrm{h}_r(a,b)=0\) for \(r<0\).
For \(a>b\) the alphabet is empty, so \(\mathrm{h}_r(a,b)\) is \(1\) for
\(r=0\) and \(0\) otherwise.  If \(a_i\leq b_i\) for all \(i\), then the
Jacobi--Trudi identity of I.~Gessel and M.~Wachs~\cite{Wachs1985} gives
\[
s_{\lambda/\mu}(\mathbf{a},\mathbf{b})
=
\det\bigl(\mathrm{h}_{\lambda_i-\mu_j-i+j}(a_j,b_i)\bigr)_{1\leq i,j\leq m}.
\]

\begin{corollary}[Stretched flagged skew Kostka coefficients]
\label{cor:flaggedSkewKostka}
Let \(\lambda/\mu\) be a skew shape with \(m\) rows, let
\(\mathbf{a},\mathbf{b}\) be weakly increasing flags with entries in
\(\{1,\dotsc,h\}\), and let \(\lambda',\mu'\in\setZ^m\) and
\(\nu,\eta\in\setZ^h\) be such
that \((N\lambda+\lambda')/(N\mu+\mu')\) is a skew shape for \(N\gg0\).  Then
\[
N\longmapsto
[\xvec^{N\nu+\eta}]\,
s_{(N\lambda+\lambda')/(N\mu+\mu')}(\mathbf{a},\mathbf{b})
\]
is eventually polynomial.
\end{corollary}

\begin{proof}
We may assume that \(a_i\leq b_i\) for all \(i\).  Indeed, the length
\(N(\lambda_i-\mu_i)+\lambda'_i-\mu'_i\) of row \(i\) is affine in \(N\), so
for \(N\gg0\) it is either always positive or always zero.  If \(a_i>b_i\)
and row \(i\) is eventually nonempty, then there are no tableaux, and the
coefficient is eventually zero.  Otherwise every row with \(a_i>b_i\) is
eventually empty, and the flag of an empty row plays no role.  So
\(\max(\mathbf{a},\mathbf{b})\), which is weakly increasing, differs from
\(\mathbf{b}\) only in eventually empty rows, and replacing \(\mathbf{b}\) by
it does not change the polynomial for \(N\gg0\).

For \(r\geq0\) and \(a\leq b\), we have
\(\mathrm{h}_r(a,b)=\pi_{b-1}\dotsm\pi_{a+1}\pi_a(x_a^r)\).  Indeed,
\(\pi_c\) fixes \(\mathrm{h}_{r-j}(a,c-1)\) and sends \(x_c^j\) to
\(\mathrm{h}_j(c,c+1)\), so
\(\pi_c\,\mathrm{h}_r(a,c)
=\sum_{j}\mathrm{h}_{r-j}(a,c-1)\,\mathrm{h}_j(c,c+1)
=\mathrm{h}_r(a,c+1)\).

Write \(\delta=(m-1,\dotsc,1,0)\).  Expanding the Jacobi--Trudi determinant
for the stretched shape gives a sum over \(\sigma\in\symS_m\) of
\(\operatorname{sgn}(\sigma)\prod_{j=1}^m\mathrm{h}_{r_j(N)}(a_j,b_{\sigma(j)})\),
where
\(r_j(N)=N(\lambda_{\sigma(j)}-\mu_j)+(\lambda'+\delta)_{\sigma(j)}-(\mu'+\delta)_j\)
is an affine function of \(N\).  For fixed \(\sigma\), either some
\(r_j(N)\) is eventually negative, or eventually nonzero on an empty range
\(a_j>b_{\sigma(j)}\), and then the term vanishes for \(N\gg0\); or every
factor is, for \(N\gg0\), the monomial \(1\) or a fixed sequence of Demazure
operators applied to the monomial \(x_{a_j}^{r_j(N)}\).  The product part of
Lemma~\ref{lem:fixedDemazureExpressions} applies to each of the finitely many
remaining terms.
\end{proof}

Cutting a cylindric diagram along a column and fixing the entries of that
column writes each cylindric Kostka coefficient as a sum of flagged skew Kostka
coefficients, with a number of terms that does not change under
stretching~\cite{AlexanderssonOguz2023x}.  The terms have exactly the form of
Corollary~\ref{cor:flaggedSkewKostka}, with \(\lambda'=0\), \(\mu'=(1^m)\)
and \(\eta\in\{0,-1\}^h\).  Hence Corollary~\ref{cor:flaggedSkewKostka}
implies that stretched cylindric Kostka coefficients are eventually
polynomial.  Since these coefficients count lattice points in dilates of a
rational polytope, they are polynomial for all \(N\geq0\), as in the proof of
Corollary~\ref{cor:pureKeyKostka}.
\endgroup

\begingroup
\subsection{Column-monotonicity}
\label{subs:columnMonotonicity}

For a partition \(\lambda\) with at most \(h\) parts and a weak
composition \(\nu\in\setN^h\) of the same size, let
\(\defin{K_{\lambda,\nu}}\) denote the number of semistandard Young
tableaux of shape \(\lambda\) and content \(\nu\).  Write
\[
\defin{E_{\lambda,\nu}(N)}\coloneqq K_{N\lambda,N\nu}
\]
for the stretched Kostka polynomial.  For \(1\leq r\leq h\), put
\(\defin{\omega_r}\coloneqq(1^r,0^{h-r})\), and let
\(\defin{\mathbf1_J}\) be the indicator vector of a subset
\(J\subseteq\{1,\dotsc,h\}\).
We propose the following strengthening of Kostka coefficient positivity.

\begin{conjecture}[Column-monotonicity]
\label{conj:kostkaColumnMonotonicity}
If \(K_{\lambda,\nu}>0\), then for every subset \(J\) of size \(r\),
\[
E_{\lambda+\omega_r,\nu+\mathbf1_J}(N)-E_{\lambda,\nu}(N)
\in\mathbb Q_{\geq0}[N].
\]
\end{conjecture}

Here the comparison is in ordinary powers of \(N\).  At each positive
integer \(N\), an injection is obtained by adding \(N\) times the
Gelfand--Tsetlin pattern of the column with entries \(J\) to a pattern
with top row \(N\lambda\) and weight \(N\nu\).

\begin{proposition}
\label{prop:columnMonotonicityImpliesPositivity}
Conjecture~\ref{conj:kostkaColumnMonotonicity} implies the Kostka case of
the King--Tollu--Toumazet positivity conjecture~\cite{KingTolluToumazet2004}.
It suffices to establish column-monotonicity when
\(\lambda_1=\dotsb=\lambda_r\).
\end{proposition}

\begin{proof}
If \(K_{\lambda,\nu}=0\), the dominance criterion gives
\(K_{N\lambda,N\nu}=0\) for every \(N\geq1\).
Otherwise, choose a tableau of shape \(\lambda\) and content \(\nu\),
and build it from its columns, from left to right.
Before adding a column of height \(r\), the first \(r\) rows all have
the same length.  Each step therefore gives one of the conjectured
coefficientwise inequalities.  Telescoping from \(E_{0,0}=1\) shows
that \(E_{\lambda,\nu}-1\) has nonnegative coefficients.
\end{proof}

\begin{question}
\label{question:keySchubertColumnMonotonicity}
Which simultaneous changes of the index and weight for key--Kostka and
Schubert--Kostka coefficients generalize the column addition in
Conjecture~\ref{conj:kostkaColumnMonotonicity} and give coefficientwise
monotonicity of their stretching polynomials?
\end{question}
For keys, we give a precise formulation in
Conjecture~\ref{conj:keyColumnMonotonicity}.
\endgroup

\section{Stretching Lehmer codes}
\label{sec:stretchingRotheDiagrams}

Let \(\defin{\symS_m}\) denote the symmetric group on \(m\) letters.
For \(w\in\symS_m\), its \defin{Rothe diagram} is
\[
\defin{D(w)}
\coloneqq
\{(i,j):1\leq i,j\leq m,\ j<w(i),\ i<w^{-1}(j)\}.
\]
We use matrix coordinates: \(i\) is the row index and \(j\) is the column
index.  Thus the cells in row \(i\) are the pairs \((i,j)\in D(w)\).  The row
lengths form its \defin{Lehmer code},
\[
\defin{\code(w)_i}
\coloneqq
\#\{j>i:w(j)<w(i)\}.
\]
Conversely, if \(\gamma=(\gamma_1,\dotsc,\gamma_h)\) is a weak composition
and \(t\) satisfies \(\gamma_i\leq h+t-i\) for all \(i\), then
\((\gamma_1,\dotsc,\gamma_h,0^t)\) is the Lehmer code of a unique
permutation in \(\symS_{h+t}\).

A permutation is \defin{\(k\)-Grassmannian} if it has no descents except
possibly at position \(k\).  For Grassmannian permutations, code stretching
has the usual interpretation as stretching a Ferrers shape.

\begin{lemma}[Grassmannian stretching]
\label{lem:grassmannianStretching}
Let \(\lambda=(\lambda_1,\dotsc,\lambda_k)\) be a partition, and let \(u\) be
the permutation with Lehmer code \((\lambda_k,\dotsc,\lambda_1)\).  Then
\(u\) is \(k\)-Grassmannian and
\(\schubertS_u(x_1,\dotsc,x_k)=\defin{s_\lambda(x_1,\dotsc,x_k)}\), the Schur
polynomial indexed by \(\lambda\).  Moreover,
\(\stretchperm{N}{u}\) is the Grassmannian permutation associated with the
stretched shape \(N\lambda\), and
\[
\schubertS_{\stretchperm{N}{u}}(x_1,\dotsc,x_k)
=
s_{N\lambda}(x_1,\dotsc,x_k).
\]
\end{lemma}

\begin{proof}
The Schubert--Schur identity for Grassmannian permutations is standard; see,
for example,~\cite{ReinerShimozono1995}.  The remaining assertion follows from
\[
\code(\stretchperm{N}{u})
=
(N\lambda_k,\dotsc,N\lambda_1).
\]
The claim follows.
\end{proof}

\begingroup
Pak and Slonim prove a stability statement for pipe-dream polytopes under code
scaling~\cite[Rem.~4.4]{PakSlonim2026x}.  For Rothe diagrams, we separate
the dependence on \(N\) from the description of the columns.

\begin{lemma}[Affine permutation entries]
\label{lem:affinePermutationEntries}
Let \(\alpha\in\setN^h\) and \(\beta\in\setZ^h\), with \(N\alpha+\beta\)
nonnegative for \(N\gg0\), and let \(w(N)\) have Lehmer code
\(N\alpha+\beta\), followed by sufficiently many zeroes.
For \(N\gg0\), the first \(h\) values of \(w(N)\) are affine functions
of \(N\), and their relative order is constant.
\end{lemma}

\begin{proof}
We use induction on the position \(i\).  The inverse Lehmer-code construction
selects the \((N\alpha_i+\beta_i+1)\)-st unused positive integer, so
\[
w(N)(i)=N\alpha_i+\beta_i+1+r
\qquad\text{for some }0\leq r<i.
\]
A candidate is selected exactly when it is unused and precisely \(r\) earlier
values are smaller.  By induction, the earlier values are eventually affine.
Each comparison with a candidate therefore stabilizes, so the selected
\(r\) is eventually constant.  Comparing the resulting affine functions
also shows that their relative order stabilizes.
\end{proof}

\endgroup

\begingroup
\section{Dominant sorting}
\label{sec:dominantSorting}

For \(1\leq i<n\), the \defin{divided difference} is
\(\defin{\partial_i(f)}\coloneqq(f-s_if)/(x_i-x_{i+1})\).  Let
\(\defin{\ell(w)}\) be the number of inversions of \(w\in\symS_n\).  Schubert
polynomials are determined by
\(\schubertS_{w_0}=x_1^{n-1}x_2^{n-2}\dotsm x_{n-1}\) for the longest element
\(w_0\in\symS_n\), and \(\schubertS_{ws_i}=\partial_i\schubertS_w\) whenever
\(\ell(ws_i)<\ell(w)\); they do not depend on \(n\)~\cite{Macdonald1991}.
Consequently, if \(v=s_{i_1}\dotsm s_{i_k}\) is a reduced word and
\(\ell(wv)=\ell(w)+\ell(v)\), then
\[
\schubertS_w=\defin{\partial_v}\schubertS_{wv},
\qquad
\partial_v\coloneqq\partial_{i_1}\dotsm\partial_{i_k}.
\]
A permutation \(u\) is \defin{dominant} if \(\code(u)\) is weakly decreasing,
and then \(\schubertS_u=\xvec^{\code(u)}\)~\cite{Macdonald1991}.

\begin{lemma}[Dominant sorting]
\label{lem:dominantSorting}
Let \(w\in\symS_n\), with \(1\leq h\leq n\), and suppose that
\(\code(w)_i=0\) for \(i>h\).  Let \(\sigma\in\symS_h\) record the relative
order of \(w(1),\dotsc,w(h)\), put \(v=\sigma^{-1}w_0\), where \(w_0\) is the
longest element of \(\symS_h\), and set \(\defin{t_r}\coloneqq w(v(r))\).
Then \(t_1>\dotsb>t_h\), and
\[
\schubertS_w=\partial_v\,\xvec^{(t_1-1,\dotsc,t_h-1)}.
\]
\end{lemma}

\begin{proof}
Put \(u=wv\).  It lists \(w(1),\dotsc,w(h)\) in decreasing order and agrees
with \(w\) after position \(h\).  The \(t_r-1\) values smaller than \(t_r\)
all lie after position \(r\), so \(\code(u)_r=t_r-1\) for \(r\leq h\), and
\(\code(u)_j=\code(w)_j=0\) for \(j>h\).  Hence \(u\) is dominant and
\(\schubertS_u=\xvec^{(t_1-1,\dotsc,t_h-1)}\).  Inversions with at most one
position in \(\{1,\dotsc,h\}\) depend only on the set of the first \(h\)
values, so \(u\) and \(w\) share them.  Within the first \(h\) positions,
\(w\) has \(\ell(\sigma)\) inversions and \(u\) has \(\binom h2\).  Since
\(\ell(v)=\binom h2-\ell(\sigma)\), we get \(\ell(u)=\ell(w)+\ell(v)\), and the
claim follows.
\end{proof}

\begin{corollary}[A fixed expression]
\label{cor:fixedSchubertExpression}
Let \(\alpha\in\setN^h\) and \(\beta\in\setZ^h\), with \(N\alpha+\beta\)
nonnegative for \(N\gg0\), and let \(w(N)\) have Lehmer code
\(N\alpha+\beta\), followed by sufficiently many zeroes.  There are a
permutation \(v\in\symS_h\) and \(a,b\in\setZ^h\) such that
\[
\schubertS_{w(N)}=\partial_v\,\xvec^{Na+b}\qquad(N\gg0),
\]
and \(Na+b\) is strictly decreasing and nonnegative for \(N\gg0\).  Since
\(\partial_i(f)=\pi_i(x_i^{-1}f)\), the polynomial \(\schubertS_{w(N)}\) is, for
\(N\gg0\), a fixed expression in the sense of
Lemma~\ref{lem:fixedDemazureExpressions}.
\end{corollary}

\begin{proof}
By Lemma~\ref{lem:affinePermutationEntries}, for \(N\gg0\) the first \(h\)
values of \(w(N)\) are affine functions of \(N\) in a constant relative order.
Thus \(\sigma\) and \(v\) in Lemma~\ref{lem:dominantSorting} are constant and
each \(t_r(N)-1\) is affine.  The identity
\(\pi_i(x_i^{-1}f)=(f-s_if)/(x_i-x_{i+1})\) is the definition of \(\pi_i\).
\end{proof}

\begin{example}[Dominant sorting for \(\mathtt{321654}\)]
\label{ex:dominant321654}
Let \(\defin{u}\coloneqq\mathtt{321654}\), so that
\(\code(u)=(2,1,0,2,1,0)\).  For \(N\geq3\), the first six values of
\(\stretchperm{N}{u}\) have relative order \(\sigma=\mathtt{531642}\), so
\(v=\mathtt{415263}=s_3s_2s_1s_4s_3s_5\), and
\((t_r-1)_r=(2N+3,2N,N+2,N,1,0)\).  For instance,
\(\stretchperm{4}{u}=(9,5,1,12,7,2,3,4,6,8,10,11)\).  Hence
\[
\schubertS_{\stretchperm{N}{\mathtt{321654}}}
=\partial_3\partial_2\partial_1\partial_4\partial_3\partial_5\,
\xvec^{(2N+3,2N,N+2,N,1,0)}
\qquad(N\geq3).
\]
\end{example}
\endgroup

\section{Stable keys and Schubert--Kostka coefficients}
\label{sec:stableKeyExpansions}
\label{subsec:stableKeyExpansions}

Schubert polynomials have nonnegative key expansions
\cite{ReinerShimozono1995,Assaf2021}.  We prove that, for an affine family
of Lehmer codes, only the key indices eventually depend on \(N\).

\begin{theorem}[Stable affine key expansion]
\label{thm:stableAffineKeyExpansion}
Fix \(\alpha\in\setN^h\) and \(\beta\in\setZ^h\), and suppose that
\(N\alpha+\beta\) is nonnegative for \(N\gg0\).  Let \(w(N)\) be the
permutation with Lehmer code \(N\alpha+\beta\), after appending sufficiently
many trailing zeroes.  There are positive integers \(c_1,\dotsc,c_t\), weak
compositions \(\gamma_1,\dotsc,\gamma_t\), and integer vectors
\(\delta_1,\dotsc,\delta_t\) of a common fixed length, all independent of
\(N\), such that \(N\gamma_j+\delta_j\) is nonnegative for large \(N\) and
\[
\schubertS_{w(N)}
=
\sum_{j=1}^t c_j\keypoly_{N\gamma_j+\delta_j}
\qquad(N\gg0).
\]
In particular, the assertion holds for \(w(N)=\stretchperm{N}{u}\).
\end{theorem}

Corollary~\ref{cor:fixedSchubertExpression} separates the growing exponent from
all other operations.  To bring it into the form of the following lemma, put
\(\defin{e_h}\coloneqq x_1\dotsm x_h\).  Since \(e_h\) is symmetric and
\(\partial_i(x_ig)=\pi_i(g)\), we have
\(e_h\,\partial_i(f)=\pi_i\bigl((\prod_{j\neq i}x_j)f\bigr)\).  The following
lemma gives a finite signed key expansion for such expressions.  We use
Schubert key positivity only at the end.

\begin{lemma}[A dominant exponent and fixed operators]
\label{lem:signedAffineKeys}
Let \(a\in\setN^h\) be dominant, let \(F\in\setZ[x_1,\dotsc,x_h]\), and let
\(\mathcal T\) be a fixed finite composition of Demazure operators and
multiplications by polynomials in \(\setZ[x_1,\dotsc,x_h]\).  There are
integers \(d_j\) and weak compositions \(\gamma_j,\delta_j\), all independent
of \(N\), such that
\[
\mathcal T(\xvec^{Na}F)
=
\sum_j d_j\keypoly_{N\gamma_j+\delta_j}
\qquad(N\gg0).
\]
The sum is finite, and each \(\gamma_j\) is a permutation of \(a\).
\end{lemma}

\begin{proof}
We first move polynomial multipliers to the right of the Demazure
operators.  The identity
\[
g\,\pi_i(f)=\pi_i\bigl((s_i g)f\bigr)
  +\bigl(g-\pi_i(s_i g)\bigr)f
\]
follows from the definition of \(\pi_i\); both multipliers on the right
are polynomials with integer coefficients.  Iterating this identity and
using the \(0\)-Hecke relations gives
\[
\mathcal T(f)=\sum_{w\in\symS_h}\pi_w(g_wf)
\]
for fixed polynomials \(g_w\).

Partition the coordinates into maximal consecutive blocks on which \(a\)
is constant, and let \(W\) be the group of permutations within these
blocks.  Products of the ordinary polynomial key bases on the blocks form
an integral basis of \(\setZ[x_1,\dotsc,x_h]\)
\cite{Lascoux1990Keys}.  Expanding each \(g_wF\) in this basis gives
\[
g_wF=\sum_{z,\delta}d_{w,z,\delta}\pi_z\xvec^\delta,
\qquad z\in W,
\]
where the sum is finite and each \(\delta\in\setN^h\) is weakly decreasing
within every block.

The monomial \(\xvec^{Na}\) is \(W\)-invariant, so multiplication by it
commutes with every \(\pi_z\) in this expansion.  Moreover, \(Na+\delta\)
is globally dominant for large \(N\): dominance holds within each block
by construction and between blocks because the slopes strictly decrease.
Hence
\[
\mathcal T(\xvec^{Na}F)
=
\sum_{w,z,\delta}d_{w,z,\delta}\pi_w\pi_z\xvec^{Na+\delta}
=
\sum_{w,z,\delta}d_{w,z,\delta}
\keypoly_{(w\star z)(Na+\delta)}.
\]
Here \(w\star z\) is the Demazure product, characterized by
\(\pi_w\pi_z=\pi_{w\star z}\).  Each index is affine with nonnegative
slope and intercept.  Equalities among finitely many affine indices
eventually stabilize, so we may collect equal keys with constant integer
coefficients.
\end{proof}

\begin{proof}[Proof of Theorem~\ref{thm:stableAffineKeyExpansion}]
By Corollary~\ref{cor:fixedSchubertExpression},
\(\schubertS_{w(N)}=\partial_v\xvec^{Na+b}\) for \(N\gg0\), with \(v\) fixed and
\(Na+b\) eventually dominant and nonnegative.  Let \(k=\ell(v)\).  Applying the
identity for \(e_h\partial_i\) \(k\) times gives
\(e_h^k\schubertS_{w(N)}=\mathcal T(\xvec^{Na+b})\), where \(\mathcal T\) is a
fixed composition of Demazure operators and multiplications by monomials.
The slope \(a\) is a dominant weak composition.  Choose \(N_0\) so that
\(N_0a+b\) is nonnegative, put \(N=M+N_0\), and apply
Lemma~\ref{lem:signedAffineKeys} with \(F=\xvec^{N_0a+b}\).  Returning to \(N\)
and collecting eventually equal indices gives a finite expansion
\(e_h^k\schubertS_{w(N)}=\sum_jd_j\keypoly_{N\gamma_j+\delta_j}\) for \(N\gg0\),
with fixed integers \(d_j\) and affine indices.

Since \(\pi_i\) commutes with multiplication by \(e_h\), we have
\(\keypoly_{\gamma+k\mathbf 1}=e_h^k\keypoly_\gamma\), where
\(\mathbf 1=(1,\dotsc,1)\).  Multiplying the nonnegative key expansion of
\(\schubertS_{w(N)}\) by \(e_h^k\) and comparing with the previous expansion,
uniqueness of the key basis shows that, for every sufficiently large \(N\), the
coefficient of \(\keypoly_{N\gamma_j+\delta_j-k\mathbf 1}\) in
\(\schubertS_{w(N)}\) is \(d_j\).  Discarding zero coefficients gives the
asserted expansion.
\end{proof}

\begin{example}[A stable key expansion]
\label{ex:stableKeyExpansion321654}
For \(u=\mathtt{321654}\), we expand the expression of
Example~\ref{ex:dominant321654} in keys.  Omitting trailing zeroes,
for \(N\geq3\) we obtain
\[
\begin{aligned}
\schubertS_{\stretchperm{N}{\mathtt{321654}}}
={}&
\keypoly_{(2N,N,0,2N,N)}
+\keypoly_{(2N,N+1,0,2N,N-1)}\\
&+\keypoly_{(2N,2N+1,0,N-1,N)}
+\keypoly_{(2N+1,N,0,2N-1,N)}\\
&+\keypoly_{(2N+1,N+1,0,2N-1,N-1)}
+\keypoly_{(2N+2,N,0,2N-2,N)}\\
&+\keypoly_{(2N+2,N+1,0,2N-2,N-1)}
+\keypoly_{(2N+2,2N-1,0,N-1,N)}.
\end{aligned}
\]
Each of the eight affine key indices has coefficient \(1\).  The bound
\(N\geq3\) is sharp: at \(N=2\), the expansion has only seven terms.
\end{example}

\begin{corollary}[Affine Schubert--Kostka polynomiality]
\label{thm:affineSchubertKostka}
Fix \(\alpha,\nu\in\setN^h\) and \(\beta,\eta\in\setZ^h\), and let \(w(N)\)
be the permutation whose Lehmer code is obtained from \(N\alpha+\beta\) by
appending enough trailing zeroes to satisfy the Lehmer-code inequalities.
Suppose that \(N\alpha+\beta\) and
\(N\nu+\eta\) are nonnegative and have the same total degree for all
sufficiently large \(N\).  Then
\[
N\longmapsto
[\xvec^{N\nu+\eta}]\schubertS_{w(N)}(\xvec)
\]
is eventually polynomial.
\end{corollary}

\begin{proof}
By Corollary~\ref{cor:fixedSchubertExpression}, \(\schubertS_{w(N)}\) is, for
\(N\gg0\), a fixed expression in the sense of
Lemma~\ref{lem:fixedDemazureExpressions}, which proves the claim.
\end{proof}

\subsection{A Grothendieck--Lascoux conjecture}

M.~Shimozono and T.~Yu proved that Grothendieck polynomials expand positively
in Lascoux polynomials~\cite{ShimozonoYu2021}.  This suggests the following
connective-\(K\) analogue of Theorem~\ref{thm:stableAffineKeyExpansion}.  We
write \(\defin{\mathfrak G_w^{(\beta)}}\) and
\(\defin{\mathfrak L_\alpha^{(\beta)}}\) for the connective-\(K\)
Grothendieck and Lascoux polynomials, respectively.  Here \(\beta\) is the
connective-\(K\) parameter, and \(\beta=0\) recovers Schubert and key
polynomials.

\begin{conjecture}[Stable Grothendieck--Lascoux expansion]
\label{conj:stableGrothendieckLascouxExpansion}
For every permutation \(u\), there are positive integers
\(c_1,\dotsc,c_t\) and nonnegative integers \(e_1,\dotsc,e_t\).  There are
also weak compositions \(\gamma_1,\dotsc,\gamma_t\) and integer vectors
\(\delta_1,\dotsc,\delta_t\) of a common fixed length.  All these data are
independent of \(N\), and \(N\gamma_j+\delta_j\) is nonnegative for large
\(N\).  Moreover,
\[
\mathfrak G_{\stretchperm{N}{u}}^{(\beta)}
=
\sum_{j=1}^t
c_j\beta^{e_j}\mathfrak L_{N\gamma_j+\delta_j}^{(\beta)}
\qquad(N\gg0).
\]
\end{conjecture}

For \(u=\mathtt{2143}\), the predicted expansion is
\[
\mathfrak G_{\stretchperm{N}{\mathtt{2143}}}^{(\beta)}
=
\mathfrak L_{(N,0,N)}^{(\beta)}
+\mathfrak L_{(N+1,0,N-1)}^{(\beta)}
+\beta\mathfrak L_{(N+1,0,N)}^{(\beta)}.
\]
Computations for every \(u\in\symS_6\) and \(1\leq N\leq6\) agree
with an affine pattern beginning no later than \(N=3\), and the displayed
formula for \(u=\mathtt{2143}\) agrees through \(N=10\).

\section{Products and Schubert structure constants}
\label{sec:products}

Stable key expansions also reduce products of Schubert polynomials to
products of affine keys.  We record the coefficient consequence needed for
the subsequent extraction of Schubert structure constants.

\begin{proposition}[Affine product Schubert--Kostka polynomiality]
\label{prop:productSchubertKostka}
Fix \(h\), and for \(1\leq j\leq s\), fix
\(\alpha^{(j)}\in\setN^h\) and \(\beta^{(j)}\in\setZ^h\).  Suppose that
\(N\alpha^{(j)}+\beta^{(j)}\) is nonnegative for \(N\gg0\), for each \(j\).
For \(N\gg0\), let \(w_j(N)\) be the permutation whose Lehmer code is obtained
from \(N\alpha^{(j)}+\beta^{(j)}\) by appending enough trailing zeroes to
satisfy the Lehmer-code inequalities.
For every fixed
\(\nu\in\setN^h\) and \(\eta\in\setZ^h\),
\[
[\xvec^{N\nu+\eta}]
\prod_{j=1}^s
\schubertS_{w_j(N)}(\xvec)
\]
is eventually polynomial in \(N\).  If the total degrees eventually disagree, the
coefficient is eventually zero.
\end{proposition}

\begin{proof}
By Corollary~\ref{cor:fixedSchubertExpression}, each factor is, for \(N\gg0\),
a fixed expression in the sense of Lemma~\ref{lem:fixedDemazureExpressions},
and the product form of that lemma applies.  The assertion about total degrees
follows from homogeneity.
\end{proof}

We now use the Schubert analogue of the usual Weyl-denominator extraction of
Littlewood--Richardson coefficients.  The point is to express a Schubert
structure constant using finitely many monomial coefficients, so that the
preceding proposition applies.  We keep the number \(h\) of variables
fixed, even though the permutations may require more letters as \(N\) grows.

For a weak composition
\(\alpha\in\setN^h\), write \(\defin{\schubertS_\alpha}\) for the Schubert polynomial
whose index has Lehmer code \(\alpha\).  Put
\(\defin{\mathbf 1}\coloneqq(1,\dotsc,1)\).  For \(\lambda\in\setZ^h\), choose an integer
\(k\) such that \(\lambda+k\mathbf 1\) is nonnegative, and define the
\defin{generalized Schubert polynomial}
\[
\defin{\schubertS_\lambda(\xvec)}
\coloneqq
(x_1\dotsm x_h)^{-k}\schubertS_{\lambda+k\mathbf 1}(\xvec).
\]
This is independent of \(k\), since adding \(\mathbf 1\) to a nonnegative
code multiplies its Schubert polynomial by \(x_1\dotsm x_h\)
\cite[Sec.~3]{Watanabe2016}.  Put
\[
\defin{\rho}\coloneqq(h-1,h-2,\dotsc,0),
\qquad
\defin{\Delta(\xvec)}\coloneqq\prod_{1\leq i<j\leq h}(x_i-x_j),
\]
and equip Laurent polynomials with the monomial pairing
\[
\defin{\langle\xvec^a,\xvec^b\rangle}
\coloneqq
\begin{cases}
1,& a=b,\\
0,& a\neq b.
\end{cases}
\]
For \(\lambda,\mu\in\setN^h\), Watanabe's duality formula
\cite[Lem.~7.4]{Watanabe2016} states that
\begin{equation}
\label{eq:watanabeDuality}
\left\langle
\schubertS_\lambda(\xvec),
\schubertS_{\rho-\mu}(\xvec^{-1})\Delta(\xvec)
\right\rangle
=
\begin{cases}
1,& \lambda=\mu,\\
0,& \lambda\neq\mu.
\end{cases}
\end{equation}

\begin{theorem}[Affine Schubert structure constants]
\label{thm:affineSchubertStructureConstants}
Fix \(\alpha,\gamma,\tau\in\setN^h\) and
\(\beta,\delta,\zeta\in\setZ^h\).  Suppose that
\(N\alpha+\beta\), \(N\gamma+\delta\), and \(N\tau+\zeta\) belong to
\(\setN^h\) for \(N\gg0\), and let \(u(N),v(N),w(N)\) be the permutations
with these respective Lehmer codes.  Then
\[
N\longmapsto c_{u(N),v(N)}^{w(N)}
\]
is eventually polynomial.
\end{theorem}

\begin{proof}
Unless the total degree of \(N\tau+\zeta\) equals the sum of the total degrees
of \(N\alpha+\beta\) and \(N\gamma+\delta\) identically in \(N\), the
coefficient is zero for all sufficiently large \(N\).  We therefore assume
that the degrees agree.
The Schubert polynomials indexed by weak compositions in \(\setN^h\) form a
basis of \(\setZ[x_1,\dotsc,x_h]\).  Pairing the Schubert expansion of
\(\schubertS_{u(N)}(\xvec)\schubertS_{v(N)}(\xvec)\) with
\eqref{eq:watanabeDuality} gives
\[
c_{u(N),v(N)}^{w(N)}
=
\left\langle
\schubertS_{u(N)}(\xvec)\schubertS_{v(N)}(\xvec),
\schubertS_{\rho-(N\tau+\zeta)}(\xvec^{-1})\Delta(\xvec)
\right\rangle.
\]
Using
\[
\Delta(\xvec)
=\sum_{\sigma\in\symS_h}
\operatorname{sgn}(\sigma)\xvec^{\sigma\rho},
\]
and the identity
\(\langle F(\xvec),G(\xvec^{-1})\xvec^a\rangle
=[\xvec^a]F(\xvec)G(\xvec)\), we obtain
\begin{equation}
\label{eq:finiteSchubertExtraction}
c_{u(N),v(N)}^{w(N)}
=
\sum_{\sigma\in\symS_h}
\operatorname{sgn}(\sigma)
[\xvec^{\sigma\rho}]
\schubertS_{u(N)}(\xvec)
\schubertS_{v(N)}(\xvec)
\schubertS_{\rho-(N\tau+\zeta)}(\xvec).
\end{equation}

Choose an integer-valued affine function
\(\defin{k(N)}\coloneqq MN+K\) such that
\(\rho-(N\tau+\zeta)+k(N)\mathbf 1\) is nonnegative for \(N\gg0\).
It suffices to take \(M>\max_i\tau_i\) and then take \(K\) sufficiently
large.  For each fixed \(N\), apply the definition of generalized Schubert
polynomials with \(k=k(N)\).  Then each summand in
\eqref{eq:finiteSchubertExtraction} is
\[
[\xvec^{k(N)\mathbf 1+\sigma\rho}]
\schubertS_{u(N)}(\xvec)
\schubertS_{v(N)}(\xvec)
\schubertS_{\rho-(N\tau+\zeta)+k(N)\mathbf 1}(\xvec).
\]
The three Schubert indices and the extracted exponent are affine in \(N\).
Proposition~\ref{prop:productSchubertKostka} shows that every summand is
eventually polynomial.  There are only \(h!\) summands, independent of \(N\),
so their signed sum is eventually polynomial.  This completes the proof.
\end{proof}

\begin{corollary}[Stretched Schubert structure constants]
\label{cor:stretchedSchubertStructureConstants}
For fixed permutations \(u,v,w\), the function
\[
N\longmapsto
c_{\stretchperm{N}{u},\stretchperm{N}{v}}^{\stretchperm{N}{w}}
\]
is eventually polynomial.
\end{corollary}

\begin{proof}
Apply Theorem~\ref{thm:affineSchubertStructureConstants} with
\((\alpha,\beta)=(\code(u),0)\),
\((\gamma,\delta)=(\code(v),0)\), and
\((\tau,\zeta)=(\code(w),0)\).  This proves the corollary.
\end{proof}

This proves the conjecture stated by I.~Pak and Z.~Slonim as
Conjecture~7.1, following A.~N.~Kirillov's earlier question
\cite[Sec.~7.1]{PakSlonim2026x}.

\begin{corollary}[Stretched Littlewood--Richardson coefficients]
\label{cor:stretchedLittlewoodRichardson}
For fixed partitions \(\lambda,\mu,\nu\), the function
\[
N\longmapsto c_{N\lambda,N\mu}^{N\nu}
\]
is eventually polynomial.
\end{corollary}

\begin{proof}
Choose \(k\) such that \(\lambda,\mu,\nu\) have at most \(k\) parts, and
realize their Schur polynomials as the Grassmannian Schubert polynomials in
Lemma~\ref{lem:grassmannianStretching}.  After code stretching, these become
\(s_{N\lambda}\), \(s_{N\mu}\), and \(s_{N\nu}\), respectively.  Hence
\(c_{N\lambda,N\mu}^{N\nu}\) is a stretched Schubert structure constant, and
the claim follows from
Corollary~\ref{cor:stretchedSchubertStructureConstants}.
\end{proof}

The stronger assertion that \(c_{N\lambda,N\mu}^{N\nu}\) is given by a
polynomial for every \(N\geq1\) was proved by H.~Derksen and
J.~Weyman~\cite{DerksenWeyman2002}.  E.~Rassart later gave a chamber-complex
proof and established a piecewise-polynomial formula in the three partition
parameters~\cite{Rassart2004}.

\begingroup
\section{Column-monotonicity for keys}
\label{sec:keyColumnMonotonicity}

We extend the column comparison of
Conjecture~\ref{conj:kostkaColumnMonotonicity} to key polynomials.
For a weak composition \(\alpha\in\setN^h\), its column at position \(c\)
has row set
\[
C=\{j:\alpha_j\geq c\},\qquad 1\leq c\leq\max_j\alpha_j.
\]
Duplicating this column replaces \(\alpha\) by \(\alpha+\mathbf1_C\).
If \(C=\{c_1<\dotsb<c_r\}\) and \(J=\{j_1<\dotsb<j_r\}\), we write
\(\defin{J\leq C}\) when \(j_i\leq c_i\) for every \(i\).
The possible weights of the column are precisely the vectors
\(\mathbf1_J\) with \(J\leq C\).

\begin{conjecture}[Key column-monotonicity]
\label{conj:keyColumnMonotonicity}
Let \(\alpha,\nu\in\setN^h\) satisfy
\([\xvec^\nu]\keypoly_\alpha>0\), and let \(C\) be a nonempty column
of \(\alpha\).  For every \(J\leq C\),
\[
[\xvec^{N(\nu+\mathbf1_J)}]\keypoly_{N(\alpha+\mathbf1_C)}
-
[\xvec^{N\nu}]\keypoly_{N\alpha}
\in\mathbb Q_{\geq0}[N].
\]
\end{conjecture}
\endgroup

\section*{Acknowledgments}

We thank Igor Pak for asking about polynomiality of stretched Schubert
coefficients; that question was the starting point for this paper.
OpenAI's Codex was used during the preparation of this paper for proof
exploration and editorial assistance.

\bibliographystyle{alpha}
\bibliography{key-schubert-kostka-ALEXANDERSSON-reduced-word}

\end{document}